\documentclass{amsart}
\usepackage{amsmath}
\usepackage{amssymb}
\usepackage{amsfonts}
\usepackage{mathrsfs}
\usepackage{graphics}
\usepackage{graphicx}

\title[Time-changes of horocycle flows]%
{Time-Changes of Horocycle Flows}

\author{Giovanni Forni} \author{Corinna Ulcigrai}

\address{Department  of Mathematics\\
  University of Maryland \\
  College Park, MD USA}
  
\address{School of Mathematics \\ University of Bristol \\ Bristol, UK}
  
\begin{document}

\def\Leb{{\mathrm{Leb}}}
\def\PP{{\mathcal{P}}}
\def\RR{{{\mathcal{R}}}}
\def\R{{\mathbb{R}}}
\def\Q{{\mathbb{Q}}}
\def\C{{\mathbb{C}}}
\def\Z{{\mathbb{Z}}}
\def\N{{\mathbb{N}}}
\def\T{{\mathbb{T}}}
\def\Re{{\operatorname{Re}}}
\def\Im{{ \operatorname{Im} }}

\def\be{\begin{equation}}
\def\ee{\end{equation}}
\def\bes{\begin{equation*}}
\def\ees{\end{equation*}}
\newcommand{\Sym}[1]{\mathcal{S} _{ #1} }
\newcommand{\ud}{\, \mathrm{d}}
\def\st{{| \ }}

\newenvironment{proofof}[2]{\begin{proof}[Proof of #1 \ref{#2}.]}{\end{proof}}
\newtheorem{thm}{Theorem}
\newtheorem{lemma}{Lemma}
\newtheorem{cor}{Corollary}
\newtheorem{defn}{Definition}
\newtheorem{prop}{Proposition}
{\theoremstyle{remark} \newtheorem{rem}[thm]{Remark}}

\begin{abstract}
We consider smooth time-changes of the classical horocycle flows on the unit tangent bundle of a
compact hyperbolic surface and prove sharp bounds on the rate of equidistribution and the rate of mixing. We then derive results on the spectrum of smooth time-changes and show that the spectrum is absolutely continuous with respect to the Lebesgue measure on the real line and that the maximal spectral type is equivalent to Lebesgue. 

\end{abstract}  
  
  \maketitle
  
\section{Introduction}
The classical horocycle flow is a fundamental example of a unipotent, parabolic (non-hyperbolic)
flow. Its dynamical properties have been studied in great detail. It is known that the flow is
minimal \cite{He}, uniquely ergodic~\cite{Fu}, has Lebesgue spectrum and is therefore strongly mixing~\cite{Pa}, in fact mixing of all orders \cite{Ma}, and has zero entropy~\cite{Gu}. 
Its finer ergodic and rigidity properties, as well as the rate of mixing, were investigated by M.~Ratner
is a series of papers \cite{Ra1}, \cite{Ra2}, \cite{Ra3}, \cite{Ra4} (for results on the rate of mixing of the geodesic as well as horocycle flows see also the paper by C. Moore \cite{Mo}). In  joint work with L.~Flaminio \cite{FlaFo}, the first author has proved precise bounds on ergodic integrals of smooth 
functions. In the case of finite-volume, non-compact surfaces, the horocyle flow is not uniquely ergodic
and the classification of invariant measures is due to Dani \cite{dani}. The asymptotic behaviour of averages along closed horocycles has been studied  by D.~Zagier~\cite{Za}, P.~Sarnak~\cite{Sa}, D.~Hejal \cite{Hj} and more recently in \cite{FlaFo} and by A.~Str{\"o}mbergsson \cite{St}. Horocycle flows on general geometrically finite surfaces have been studied by M.~Burger~\cite{Bur}.  

Not much is known for general smooth parabolic flows, not even for smooth perturbations
of classical horocycle flows in the compact case. In fact, even the dynamics of non-trivial 
smooth time-changes is poorly understood. Our paper addresses the latter question. By the
classification of horocycle invariant distributions \cite{FlaFo} and by the related results on
asymptotic of ergodic averages   for classical horocycle flows (see \cite{FlaFo}, \cite{BuFo}), it is 
known that smooth time-changes which are measurably trivial form a subspace of countable
codimension, so that the generic smooth time-change is not even measurably conjugate
to the horocycle flow.  It is therefore interesting to know, perhaps as a step towards a better 
understading of parabolic dynamical systems, to what extent the  dynamical properties of the 
horocycle flow persist after a smooth time-change. The most important result to date is the proof by B.~Marcus more than thirty years ago that all time-changes satisfying a mild differentiability conditions are mixing \cite{Ma}. Marcus results generalized earlier work by Kushnirenko who proved mixing for all time-changes with sufficiently small derivative in the geodesic direction \cite{Ku}. 

A.~Katok and  J.-P.~Thouvenot have conjectured that  ``Any flow obtained by a sufficiently smooth time change from a horocycle flow has countable Lebesgue spectrum'' (see \cite{KT}, Conjecture 6.8). 
In fact, the question on the spectral type of smooth time changes of horocycle flows was already
asked in Kushnirenko's paper \cite{Ku}. There the author is able to prove the relative absolute continuity of 
the spectrum of (restricted) smooth perturbations of skew-shifts, but cannot extend his results to 
time-changes of horocyle flows.
In our paper we prove sharp bounds on the rate of equidistribution and mixing
of smooth time-changes of the classical horocycle flow on the unit tangent bundle of a
compact hyperbolic surface (see Theorem \ref{thm:quant_equidist} and Theorem \ref{thm:mixingbound2} 
in Sections \ref{sec:equidistribution} and \ref{sec:mixing} respectively). We then derive results on the spectrum of smooth time-changes (in Section \ref{sec:spectrum}),
most notably we prove that 
the spectrum is absolutely continuous with respect to the Lebesgue measure (Theorem \ref{thm:abs_cont_spectrum}, in Section \ref{subsec:ac}). 
We finally prove that the maximal spectral type is indeed equivalent to Lebesgue (Theorem \ref{thm:Max_spec_type}, in Section \ref{subsec:spectral_type}).

The guiding idea of our work is that Marcus' mixing mechanism can be made quantitative
by the more recent quantitative results on the rate of equidistribution for horocycle flows
(see \cite{Bur}, \cite{FlaFo}, \cite{BuFo}). In fact, Marcus argument is based on the equidistribution
of long {\it horocycle-like arcs}, that is, arcs which are long in the horocycle direction and bounded
in the complementary directions. Sharp results on the rate of equidistributions of horocycle-like arcs
were recently obtained in \cite{BuFo} as a refinement of earlier results for horocycle arcs
\cite{Bur}, \cite{FlaFo}.  Finally, our estimates on the rate of mixing for time-changes
would be far from optimal, and definitely too weak to derive any significant spectral result,
without a key {\it bootstrap trick}. Thanks to this bootstrap trick we can prove that 
decay of correlations of the horocycle flows are indeed stable under any smooth 
time-change. 

Spectral results are derived from square-mean bounds on twisted ergodic integrals
of smooth functions which are equivalent to bounds on the Fourier transform of the
spectral measures. A well-known difficulty in this approach is that the decay of correlations
of a general smooth function under the horocycle flow is not square-integrable, so that
it would seem hopeless to prove absolute continuity of the spectrum in this way. 
However, our results on decay of correlations of time-changes are precise enough,
thanks to the bootstrap trick, to give optimal, and hence square-integrable, decay of 
correlations for smooth coboundaries. Once it is established that all smooth coboundaries
have absolutely continuous spectral measures, it  follows (for instance by a density argument) 
that the spectrum is purely absolutely continuous. Our estimates on decay of correlations
of coboundaries are also crucial in the proof that the maximal spectral type is Lebesgue.

Let us remark that while in this paper we only deal with horocycle flows for compact hyperbolic surfaces,  most 
of the methods and results can presumably be extended to the non-compact, finite volume case with appropriate modifications.

\subsubsection*{Structure of the paper} In Section \ref{sec:time_changes} we introduce basic definitions, notation and properties of smooth time-canges of the classical horocycle flow. In Section \ref{sec:equidistribution} we recall the results on the invariant distributions for the classical horocycle flow from \cite{FlaFo} and, from the  results on the asymptotics of ergodic integrals in \cite{FlaFo,BuFo}, we derive analogous results for smooth-time changes (Theorem \ref{thm:quant_equidist}). These quantitative equidistribution results are used in  Section \ref{sec:mixing} together with a key bootstrap trick to make quantitative Marcus'mixing argument for smooth-time changes and derive the quantitative mixing result  in Theorem \ref{thm:mixingbound2}. In Section \ref{sec:L2estimates} we prove mean-square bounds on twisted ergodic integrals of smooth 
functions. Finally, in Section \ref{sec:spectrum} we prove our main spectral results. We first prove a local estimate (Theorem \ref{thm:spectral_measures} in \S~\ref{subsec:local}), then absolute continuity if the spectrum (Theorem \ref{thm:abs_cont_spectrum} in \S~\ref{subsec:ac}) and finally that the maximal spectral type is Lebesgue (Theorem \ref{thm:Max_spec_type} in \S~\ref{subsec:spectral_type}).

\section{Time-changes of horocycle flows}\label{sec:time_changes}
Let $\{U,V, X\}$ the basis of the Lie algebra $\mathfrak {sl}(2,\R)$ of $PSL(2,\R)$ given
by the generators $U$  and $V$ of the stable and unstable horocycle flows and by the generator 
$X$ of the geodesic flow, respectively. The following commutation relations 
hold:
\begin{equation}
\label{eq:commrel}
[U,V] =2X \,, \quad [X,U]= U \,,  \quad  [X,V] = -V  \,.
\end{equation}
Let  $\{h^U_t\}$ and $\{h^V_t\}$ denote respectively the stable and unstable horocyle flows and let
$\{\phi^X_t\}$ denote the geodesic flow on a compact homogeneous space $M:= \Gamma \backslash PSL(2,\R)$. They are defined respectively by the multiplicative action on the right of the $1$-parameter 
subgroups of the group $PSL(2,\R)$ listed below:
\begin{equation}
\label{eq:homflows}
  \{ \exp (tU) \}_{t\in\R}  \,, \quad  \{ \exp (tV) \}_{t\in\R} \,, \quad  \{ \exp (tX) \}_{t\in\R}  \,.
\end{equation}
A smooth time-change of the (stable) horocycle flow is a flow $\{h^\alpha_t\}$ on $M$ defined as follows.
Let $\tau:M\times \R \to \R$  be a smooth cocycle over the flow $\{h^U_t\}$, that is, a
function with the property that
$$
\tau(x, t+ t') = \tau(x,t) + \tau( h^U_t(x), t') \,, \quad \text{ for all }  (x,t,t') \in M\times \R^2\,.
$$
We denote by $\alpha:M \to \R^+$ the infinitesimal generator of the cocycle $\tau: M\times \R \to \R$,
that is, the function defined as follows:
$$
\alpha (x) := \frac{\partial\tau}{ \partial t}(x,t) \vert _{t=0} \,, \quad \text{ for all } x\in M\,.
$$
The time-change $\{h^\alpha_t\}$  is the flow on $M$ generated by the smooth vector
field 
\begin{equation}
U_\alpha =: U/\alpha \,.
\end{equation}
One can check that $\{h^\alpha_t\}$  is given by the formula
\begin{equation}
h^\alpha_{\tau(x,t)}(x) := h^U_t(x)\,, \quad  \text{ for all }  (x,t) \in M\times \R.
\end{equation}
The flow $\{h^\alpha_t\}$ preserves the (smooth) volume form $\text{\rm vol}_\alpha:= \alpha \text{\rm vol}$;
in fact 
$$
\mathcal L_{U_\alpha} \text{\rm vol}_\alpha= \imath_{U_\alpha} \text{\rm vol}_\alpha =
 \imath_U  \text{\rm vol} = \mathcal L_U \text{\rm vol}=0\,.
$$
We will assume below that the function $\alpha:M\to \R$ is everywhere strictly positive
and is normalized so that
\begin{equation}
\label{eq:normalization}
\int_M \text{\rm vol}_\alpha = \int_M \alpha \,\text{\rm vol} =1 \,.
\end{equation}

The time-change $\{h^\alpha_t\}$ is parabolic, in fact the infinitesimal divergence of trajectories
is at most quadratic with respect to time (as is the case for the standard horocycle flow).
The tangent flow $\{Dh^\alpha_t\}$ on $TM$ is described as follows.

\begin{lemma} 
\label{lemma:tangent_flow}
The tangent flow $\{Dh^\alpha_t\}$ on $TM$ is given by the following formulas:
\begin{equation}
\label{eq:tangent_flow}
\begin{aligned}
Dh^\alpha_t (V) &= [ \int_0^t  \left( (\int_0^\tau \frac{1}{\alpha} \circ h^\alpha_u \,du)  (\frac{X\alpha}{\alpha} -1)\circ h^\alpha_\tau +  \frac{V\alpha}{\alpha} \circ h^\alpha_\tau\right) \,d\tau]\, U_\alpha \circ h^\alpha_t  \\ 
&+  V\circ h^\alpha_t  +  [\int_0^t \frac{1}{\alpha} \circ h^\alpha_\tau \,d\tau]  \,X\circ h^\alpha_t \,; \\
Dh^\alpha_t (X) &= [\int_0^t (\frac{X\alpha}{\alpha} -1)\circ h^\alpha_\tau \,d\tau] \,U_\alpha\circ h^\alpha_t  + X\circ h^\alpha_t \,; \\
\end{aligned}
\end{equation}
\end{lemma}
\begin{proof} For any vector field $W$ on $M$, let us write
$$
(h^\alpha_t)_\ast (W ) = a_t U_\alpha + b_t V  +c_t X\,.
$$
By the commutation relation
$$
[U_\alpha, X] =  (\frac{X\alpha}{\alpha} -1) U_\alpha\, \quad \text{ and } \quad  [U_\alpha, V] = 
 X/\alpha +  \frac{V\alpha}{\alpha} U_\alpha\,,
$$
hence 
$$
\frac{d}{dt} (h^\alpha_t)_\ast (W) =  [ c_t (\frac{X\alpha}{\alpha} -1) \circ h^\alpha_t 
+ b_t  \frac{V\alpha}{\alpha}\circ h^\alpha_t  ] U_\alpha + (\frac{b_t}{\alpha}\circ h^\alpha_t )  X \,.
$$
It follows that the function  $(a_t, b_t ,c_t)$ satisfies the following system of O.D.E.'s:
$$
\begin{cases} 
 \frac{da_t}{dt} &= c_t (\frac{X\alpha}{\alpha} -1) \circ h^\alpha_t 
+ b_t  \frac{V\alpha}{\alpha}\circ h^\alpha_t \,; \\ 
 \frac{db_t}{dt}&=0  \,; \\
 \frac{dc_t}{dt}&= \frac{b_t}{\alpha}\circ h^\alpha_t \,.
 \end{cases}
$$
If $W=V$,  the initial condition is $(a_0, b_0, c_0) = (0,1,0)$, hence the unique solution of the Cauchy
problem is given by the functions:
$$
\begin{aligned}
&a_t =  \int_0^t  \left( (\int_0^\tau \frac{1}{\alpha} \circ h^\alpha_u \,du)  (\frac{X\alpha}{\alpha} -1)\circ h^\alpha_\tau 
+  \frac{V\alpha}{\alpha} \circ h^\alpha_\tau\right) \,d\tau \,,  \\  &b_t \equiv 1 \,, \quad c_t= \int_0^t \frac{1}{\alpha} \circ h^\alpha_\tau \,d\tau\,.
\end{aligned}
$$
If $W=X$, the initial condition is $(a_0, b_0, c_0) = (0,0,1)$, hence the unique solution of the Cauchy
problem is given by the functions:
$$
 a_t = \int_0^t (\frac{X\alpha}{\alpha} -1)\circ h^\alpha_\tau \,d\tau \,, \quad b_t \equiv 0 \,, \quad c_t \equiv 1\,;
$$
Formula~(\ref{eq:tangent_flow}) is therefore proved.
\end{proof}

\section {Cohomological equation and quantitative equidistribution}\label{sec:equidistribution}

It is a general fact that all properties of the cohomological equation and of the asymptotics of ergodic integrals  for all smooth time-changes of any smooth flow can be read from the corresponding
properties of the flow itself.

\smallskip
Let  $L^2(M):= L^2(M, \text{\rm vol})$ denote the Hilbert space of square-integrable function
with respect to the standard volume form and , for any $r \geq 0$ let $W^r(M)\subset L^2(M)$ denote 
the standard Sobolev spaces on the compact manifold $M$ and let $W^{-r}(M)$ denote the dual spaces. 

Let $L^2(M, \text{\rm vol}_\alpha)$ denote the Hilbert space of square-integrable function with respect 
to the $\{h^\alpha_t\}$-invariant volume form and let
$$
L^2_0(M,\text{\rm vol}_\alpha):= \{ f\in L^2(M) \vert  \int_M f \text{\rm vol}_\alpha =0\}\,.
$$
Let $ \mathcal D'(M)$ be the space of distributions on $M$.  For any distribution $D\in \mathcal D'(M)$, let $D_\alpha$ be the distribution defined as follows:
$$
D_\alpha (f) := D (\alpha f ) \,, \quad \text{ for all } f \in C^\infty(M)\,.
$$
The distribution $D_\alpha$ is well-defined and belongs 
to the dual Sobolev space $W^{-r}(M)$ whenever $\alpha \in W^r(M)$ 
and $D\in W^{-r}(M)$ for any $r>3/2$. In fact,   the Sobolev space $W^r(M)$, endowed with the standard structure of Hilbert space and with the standard product of functions, is a Banach algebra for $r>\dim(M)/2$, which is here the case since $M$ is a $3$-dimensional manifold. The subspace $\mathcal I^{-r}_\alpha(M)\subset W^{-r}(M)$ of invariant distributions for the time-change $\{h^\alpha_t\}$ can be described in terms of the subspace  $\mathcal I^{-r}_U(M)\subset W^{-r}(M)$  of invariant distributions for the horocycle flow (described in \cite{FlaFo}).

\begin{lemma} Let $r>3/2$ and let $\alpha \in W^r(M)$. The following holds:
$$
\mathcal I^{-r}_\alpha(M):= \{  D_\alpha \vert  D\in \mathcal I^{-r}_U(M)\}\,.
$$
\end{lemma}
\begin{proof}  By the algebra property of the Sobolev space $W^r(M)$ for $r>3/2$,
for any non-vanishing function $\alpha \in W^r(M)$, the map $D \to D_\alpha$ is an 
automorphism of the Sobolev space $W^{-r}(M)$ (it is continuous, invertible with
continuous inverse). By definition, a distribution $D_\alpha
\in\mathcal I^{-r}_\alpha(M)$, that is, the distribution $D_\alpha\in W^{-r}(M)$ is invariant  for the time-change 
$\{h^\alpha_t\}$ of generator $U_\alpha=U/\alpha$,  if and only if 
$$
D_\alpha ( U_\alpha f) = D_\alpha (Uf / \alpha) =D(Uf)= 0\,,  \quad \text{ \rm for all }\, f\in W^{r+1}(M)\,,
$$
if and only if the distribution $D\in  W^{-r}(M)$ is invariant under the horocyle flow
$\{h^U_t\}$, that is, if and only if $D\in \mathcal I^{-r}_U(M)$, as stated.
\end{proof}
Let us recall that we say that a function $f$ on $M$ is a coboundary for the flow $\{h^\alpha_t\}$ if there exists a 
function $u$ on $M$, called the transfer
function, such that $U_\alpha f = u$. 
The subspace of coboundaries for the time-changes of the horocycle flow is  described by the 
following dictionary.

\begin{lemma} A function $f$ on $M$ is a coboundary for the time-change $\{h^\alpha_t\}$ with transfer
function $u$ on $M$ if and only if the function $\alpha f $ is a coboundary for the flow $\{h^U_t\}$
with transfer function $u$ on $M$. 
\end{lemma}
\begin{proof}
It follows immediately from the definition of coboundary recalling that by definition $U_\alpha= U/\alpha$.
\end{proof}
By the above lemmas the theory of the cohomological equation for time-changes is reduced to
that for the classical horocycle flow, developed in \cite{FlaFo}.
We state the main results below for the convenience of the reader.

\smallskip
By the theory of unitary representations for $SL(2,\R)$ (see \cite{Ba}, \cite{GF}, \cite{GN}), 
the Sobolev spaces $W^r(M)$ split as direct sums of irreducible sub-representations and each irreducible sub-representations characterized  up to unitary equivalence by the spectral value 
$\mu \in \sigma(\square)$ of the restriction of the Casimir operator $\square$ (a normalized 
generator of the center of the enveloping algebra),  that is, for all $r\in \R$, the following 
splitting holds:
\begin{equation}
\label{eq:splittingone}
W^r(M) =  \bigoplus_{\mu\in \sigma(\square)}  W^r(H_\mu) 
\end{equation}
Non-trivial irreducible unitary representations belong to three different series: the {\it principal series}  
($\mu \geq 1/4$), the complementary series ($0<\mu<1/4$) and discrete series ($\mu \leq 0$). 
We recall that the positive spectral values of the Casimir operator coincide with the eigenvalues
of the Laplace-Beltrami operator on the (compact) hyperbolic surface, while the non-positive
spectral values are given by the set of non-positive integers $\{ -n^2 + n\vert n\in \Z^+\}$. 

\smallskip
Let $\alpha \in W^r(M)$ for any $r>3/2$ be any strictly positive function. For every Casimir
parameter $\mu\in \R$, let
$$
W_\alpha^r (H_\mu) :=   \{ f \in L^2(M, \text{\rm vol}_\alpha ) \vert    \alpha f \in W^r(H_\mu)\}\,.
$$
By the above splitting~(\ref{eq:splittingone}), there is also a splitting 
\begin{equation}
\label{eq:splittingtwo}
W^r(M) = \bigoplus_{\mu\in \sigma(\square)}  W^r_\alpha(H_\mu) 
\end{equation}

For every $\mu\in \sigma(\square)$, let $\mathcal I^{-r}_U(H_\mu):= \mathcal I^{-r}_U(M)\cap W^{-r}
(H_\mu)$ and let  $\mathcal I^{-r}_\alpha(H_\mu)$ be the distributional space defined as 
$$
 \mathcal I^{-r}_\alpha(H_\mu) := \{  D_\alpha \vert  D\in \mathcal I^{-r}_U(H_\mu)\}\,.
$$
\begin{thm} Let $\alpha \in W^r(M)$ for any $r>3/2$.  The space of $U_\alpha$-invariant
distributions has a splitting
$$
\mathcal I^{-r}_\alpha(M)=  \bigoplus_{\mu \in \sigma(\square)} \mathcal I^{-r}_\alpha(H_\mu)\,.
$$
The subspace $\mathcal I^{-r}_\alpha(H_\mu)$ has dimension $2$ for all irreducible sub-representations of the principal and complementary series ($\mu >0$) and for irreducible sub-representations of the discrete series it has dimension $1$ if  $r > \frac{1+\sqrt{1-4\mu}}{2}$
and is trivial  otherwise.

For every function $f \in W_\alpha^r (H_\mu)$,  the cohomological equation $U_\alpha u= f$
has a unique solution $u\in H_\mu \subset L^2(M)$ if and only if 
$f \in \text{\rm Ann} [I^{-r}_\alpha(H_\mu)]$. In case a solution $u\in H_\mu$ exists then 
$u \in W^s(H_\mu)$ for all $s<r-1$ and the following a priori bound holds: there exists
a constant $C_{s,r}>0$, indipendent on $\mu\in \sigma(\square)$, such that 
$$
\Vert u \Vert_s \leq  C_{s,r} \Vert  \alpha f\Vert_r\,.
$$
\end{thm}

\bigskip
The quantitative equidistribution for  
time-changes, and in fact the complete asymptotics of ergodic averages,  can also be derived from the corresponding results for the classical horocycle flow, derived in \cite{FlaFo} and \cite{BuFo}.  By change of variable 
there exists a function $T: M\times \R \to \R$ such that, for any function $f$ on $M$,
\begin{equation}
\begin{aligned}
\int_0^{\mathcal T} f \circ h^\alpha_\tau (x) d\tau &= \int_0^{T(x,\mathcal T)} f \circ h^\alpha_{\tau(x,t)} (x) 
\frac{\partial\tau}{\partial t} (x,t) dt \\&=   \int_0^{T(x,\mathcal T)} (\alpha f) \circ h^U_t (x) dt \,. 
\end{aligned}
\end{equation}
By the above formula for the constant function $f=1$, it follows that the function $T: M\times \R \to \R$
is given by the following identity: for all $(x, \mathcal T) \in M\times \R$,
\begin{equation}
\label{eq:timechangefunction}
\mathcal T =  \int_0^{T(x,\mathcal T)}  \alpha \circ h^U_t (x)\,dt \,.
\end{equation}
The asymptotics of the function $T: M\times \R \to \R$ for large $\mathcal T\in \R$, uniformly with
respect to $x\in M$, can be derived by the quantitative equidistribution result of M.~Burger \cite{Bur}
(see also  \cite{FlaFo}). 

\smallskip
Let $\mu_0>0$ be the smallest strictly positive eigenvalue of the Casimir operator (that, as we remarked, coincide with the smallest  eigenvalue for the hyperbolic Laplacian on the
compact surface  $\Gamma\backslash \{z\in \C\vert \,\Im(z)>0\}$). Let $\nu_0 \in [0, 1)$ and
$\epsilon_0 \in \{0,1\}$ be the parameters defined as follows: 
\begin{equation}
\nu_0 := 
\begin{cases}   
\sqrt{1-4\mu_0} \,, \quad &\text{ \rm if } \,  \mu_0 <1/4\,, \\
0  \,, \quad &\text{ \rm if } \,  \mu_0 \geq 1/4\,;
\end{cases} 
\quad \quad
\epsilon_0 := 
\begin{cases}   
0 \,, \quad &\text{ \rm if } \,  \mu_0 \not =1/4\,, \\
1   \,, \quad &\text{ \rm if } \,  \mu_0 = 1/4\,. 
\end{cases} 
\end{equation}

\begin{lemma} 
\label{lemma:timechangefunction}
For any $r>3$, there exists a constant $C_r>0$  such that the following estimate holds. Let 
$\alpha \in W^r(M)$. For all $(x,\mathcal T)\in M\times\R^+$, 
$$
\vert T(x, \mathcal T) - \mathcal T \vert  \leq  C_r \Vert \alpha \Vert_r 
 \mathcal T^{\frac{1+\nu_0}{2}} (1+ \log^+ \mathcal T)^{\epsilon_0}\,.
$$
\end{lemma}
\begin{proof} Let us assume $\mu_0 \not =1/4$. The argument in the case  $\mu_0=1/4$ is similar. 
By the normalization condition~(\ref{eq:normalization}), it follows from the identity~(\ref{eq:timechangefunction}) and from \cite{FlaFo}, Theorem 1.5, that there exists  a constant 
$C'_r>0$ such that, for all $(x,\mathcal T) \in M\times \R$, 
\begin{equation}
\label{eq:timedifference}
\vert \mathcal T - T(x,\mathcal T) \vert  \leq C'_r \Vert \alpha \Vert_r T(x,\mathcal T)^{\frac{1+\nu_0}{2}}\,.
\end{equation}
Thus for any $v>0$ there exists $\mathcal T_v>0$ such that, for all $x\in M$ and 
$\mathcal T \geq \mathcal T_v$,
$$
\vert \frac{ \mathcal T} {T(x,\mathcal T)} -1  \vert  \leq   v \,,
$$
hence $T(x,\mathcal T) \leq  \mathcal T / (1-\sigma)$. Thus by formula~(\ref{eq:timedifference}),
for all $x\in M$ and for all $\mathcal T \geq \mathcal T_\sigma$, 
$$
\vert \mathcal T - T(x,\mathcal T) \vert  \leq \frac{C'_r}{(1-v)^{\frac{1+\nu_0}{2}}} 
 \Vert \alpha \Vert_r \mathcal T^{\frac{1+\nu_0}{2}}\,.
$$
The argument is thus completed. 
\end{proof}

By Lemma~\ref{lemma:timechangefunction}, the asymptotics of ergodic integrals for smooth time-changes of horocycle flows is entirely analogous  and can be derived from the corresponding results 
for horocycle flows, proved in \cite{FlaFo} (see also \cite{BuFo}).

For any $\mu >0$, let $\nu_\mu \in \R$ and $\epsilon_\mu\in \{0,1\}$ be defined as follows:
$$
\nu_\mu :=  \begin{cases}   
\sqrt{1-4\mu} \,, \quad &\text{ \rm if } \,  \mu <1/4\,, \\
0   \,, \quad &\text{ \rm if } \,  \mu \geq 1/4\,;
\end{cases} 
\quad  \epsilon_\mu:= \begin{cases}   
0 \,, \quad &\text{ \rm if } \,  \mu\not= 1/4\,, \\
1   \,, \quad &\text{ \rm if } \,  \mu = 1/4\,.
\end{cases} 
$$

\begin{thm}
\label{thm:quant_equidist}
For any $r>3$, there exists a constant $C_r>0$ such that the following holds.  Let
$\alpha \in W^r(M)$. 
Let $H_\mu$ be an irreducible sub-representation of the principal or complementary
series ($\mu>0$). There exists a basis $\{D^+_{\alpha,\mu}, D^-_{\alpha,\mu}\} \subset 
 \mathcal I^{-r}_\alpha(H_\mu)$ such that for any function $f\in W^r_\alpha(H_\mu)$ and 
for all  $(x,\mathcal T)\in  M\times\R^+$,
\begin{equation}
\begin{aligned}
\vert  \int_0^{\mathcal T}  f\circ h^\alpha_\tau (x) d\tau \vert    &\leq C_r 
\vert D^-_{\alpha,\mu}(f) \vert \mathcal T^{\frac{1+ \nu_\mu}{2}}  \\ &+  C_r \left(
\vert D^+_{\alpha,\mu}(f) \vert \mathcal T^{\frac{1- \nu_\mu}{2}} (1+\log^+{\mathcal T})^{\epsilon_\mu}  + \Vert f \Vert_r\right)\,.
\end{aligned}
\end{equation}
 Let $H_\mu$ be an irreducible sub-representation of the discrete
series ($\mu\leq 0$). For any function $f\in W^r_\alpha(H_\mu)$ and  for all  $(x,\mathcal T)\in  
M\times\R^+$,
\begin{equation}
\vert  \int_0^{\mathcal T}  f\circ h^\alpha_\tau (x) d\tau \vert  \leq C_r \Vert f\Vert_r 
(1+ \log^+{\mathcal T})\,.
\end{equation}
\end{thm}

\bigskip

\section {Quantitative Mixing}\label{sec:mixing}
In this section we show that the following result on quantitative mixing for smooth time-changes of horocycle flows can be derived by combining quantitative equidistribution results
with B.~Marcus' proof of mixing 
\cite{Ma}. 

Let us denote $\Vert \cdot \Vert _X$ the graph norm of the densely defined 
Lie derivative operator $\mathcal L_X:L^2(M) \to L^2(M)$, that is, for all
functions $g\in L^2(M)$ which belong to the maximal domain $\text{\rm dom}(\mathcal L_X)
\subset L^2(M)$ of $\mathcal L_X$,
$$
\Vert g \Vert_X := (\Vert g \Vert_0^2 + \Vert Xg \Vert_0^2)^{1/2}\,.
$$

\begin{thm} 
\label{thm:mixingbound2}
For any $r> 11/2$ and for any $\alpha \in W^r(M)$, there exists  a constant
$C''_r(\alpha)>0$ such that the following holds. For any  zero-average function $f\in W^r(M) \cap L^2_0(M, \text{\rm vol}_\alpha)$, for any function $g\in \text{\rm dom}(\mathcal L_X)$  and for any $t\geq 1$,
$$
\vert <f\circ h^\alpha_t, g>_{L^2(M, \text{\rm vol}_\alpha)}  \vert \leq  
C_r( \alpha) \Vert f \Vert_r \Vert g\Vert_X
\,t^{-\frac{1-\nu_0}{2}} (1+\log t)^{\epsilon_0} \,. 
$$
\end{thm}
The rest of this section is devoted to the proof of Theorem \ref{thm:mixingbound2}. The key idea of Marcus' method is to consider the push-forward under the flow of
a geodesic arc. 

Let $\sigma \in \R^+$ and $(x,t) \in M\times \R$. Let $\gamma^\sigma_{x,t}:[0,\sigma] \to M$ be the parametrized path defined as follows:
\begin{equation}
\label{eq:pushforwardpaths}
\gamma^\sigma_{x,t} (s) :=  h^\alpha_t \circ \phi^X_s(x)\,, \quad \text{ \rm for all } s\in [0,\sigma] \,.
\end{equation}
We begin by computing the velocity of the path $\gamma^\sigma_{x,t}$ and its length. Let
\begin{equation}
\label{eq:velocityfunction}
v_t(x,s) :=  \int_0^t  (\frac{X\alpha}{\alpha} -1) \circ h^\alpha_\tau \circ \phi^X_s (x) d\tau
\end{equation}
\begin{lemma} 
\label{lemma:velocity}
The following identity holds for all $(x,t,s) \in M\times \R \times [0,\sigma]$:
$$
\frac{d\gamma^\sigma_{x,t} }{ds} (s) :=   v_t(x,s) U_\alpha ( \gamma^\sigma_{x,t}(s)) \,+\, 
X( \gamma^\sigma_{x,t}(s) ) \,.
$$
\end{lemma}
\begin{proof} The velocity of the geodesic path $\{\phi^X_s(x) \vert s\in [0,\sigma]\}$ is given at all 
points by the geodesic vector field $X$ on $M$, hence
$$
\frac{d\gamma^\sigma_{x,t} }{ds} (s) = Dh^\alpha_t ( X ) \circ \phi^X_s(x)\,.
$$
The formula for the velocity of the path 
$\gamma^\sigma_{x,t}$ then follows from Lemma~\ref{lemma:tangent_flow}.
\end{proof}

From the quantitative equidistribution result for the flow $\{h^\alpha_t\}$ we derive the following
estimate on the velocity function.

\begin{lemma} 
\label{lemma:velocitybound}
For any $r>4$, there exists a constant $C_r>0$  such that the following holds. Let 
$\alpha \in W^r(M)$. For all $(x,s)\in M\times [0,\sigma]$ and for all $t > 0$,
$$
\vert \frac{v_t(x,s)}{t} + 1 \vert  \leq  C_r \Vert \alpha \Vert_r \, t^{-\frac{1-\nu_0}{2}} 
(1+\log^+ t)^{\epsilon_0}\,.
$$
\end{lemma}
\begin{proof} The function $v_t(x,s)/t$ is given by an ergodic average along the trajectories of the time-change $\{h^\alpha_t\}$  (evaluated at the point 
$\phi^X_s(x)\in M$) of the function $X\alpha/\alpha -1$ which has average equal to $-1$ with respect to the $\{h^\alpha_t\}$-invariant volume $\text{\rm vol}_\alpha$.
In fact, the latter is given by the formula $\text{\rm vol}_\alpha = \alpha \text{\rm vol}$, hence by the normalization
condition~(\ref{eq:normalization}) and the invariance of the volume $\text{\rm vol}$ under the geodesic flow,
$$
\int_M  (\frac{X\alpha}{\alpha} -1) \,\text{\rm vol}_\alpha = \int_M  (X\alpha -\alpha) \,\text{\rm vol}= -1\,.
$$
The result then follows from the quantitative equidistribution theorem stated above (see Theorem~\ref{thm:quant_equidist}).
\end{proof}

We then estimate the asymptotics (as $t \to +\infty$) of the integral 
$$
\int_0^\sigma  (f\circ h^\alpha_t \circ \phi^X_s)(x) ds\,.
$$
Let $\hat U_\alpha$ be the $1$-form on $M$ uniquely defined by the conditions
$$
\imath_{U_\alpha}  \hat U_\alpha=1 \quad \text{ \rm and } \quad \imath_X \hat U_\alpha =
 \imath_V \hat U_\alpha =0\,.
$$
\begin{lemma} 
\label{lemma:mixingformula1}
For any continuous function $f$ on $M$, for all $\sigma>0$ and $(x,t)\in M\times\R$,
$$
\int_0^\sigma  (f\circ h^\alpha_t \circ \phi^X_s)(x) ds  = -\frac{1}{t} \int_{\gamma^\sigma_{x,t}}  f \hat U_\alpha   
+ \int_0^\sigma f\circ h^\alpha_t \circ \phi^X_s(x) (  \frac{v_t(x,s)}{t} +1) ds \,.
$$
\end{lemma}
\begin{proof} By Lemma~\ref{lemma:velocity} and by the above definitions,
$$
\int_{\gamma^\sigma_{x,t}}  f \hat U_\alpha  = \int_0^\sigma   f \circ h^\alpha_t \circ \phi^X_s (x) v_t(x,s)  ds\,.
$$
The formula follows immediately.
\end{proof}

The above formula can be refined by integration by parts:
\begin{lemma} 
\label{lemma:mixingformula2}
For any continuous function $f$ on $M$, for all $\sigma>0$ and $(x,t)\in M\times\R$,
$$
\begin{aligned}
\int_0^\sigma  f\circ h^\alpha_t \circ \phi^X_s(x) ds  &= -\frac{1}{t} \int_{\gamma^\sigma_{x,t}}  f \hat U_\alpha   
+   (  \frac{v_t(x,\sigma)}{t} +1)  \int_0^\sigma  f\circ h^\alpha_t \circ \phi^X_s(x) ds  \\
&- \frac{1}{t} \int_0^\sigma \frac{\partial v_t}{\partial s} (x,S) 
[\int_0^S f\circ h^\alpha_t \circ \phi^X_s(x) ds]  dS \,.
\end{aligned}
$$
\end{lemma}

\begin{lemma}
\label{lemma:velocityderivativeformula}
For all $s>0$ and all $(x,t)\in M\times \R^+$, we have:
$$
\begin{aligned}
\frac{\partial v_t}{\partial s}(x,s)&=v_t(x,s) (\frac{X\alpha}{\alpha}) \circ h^\alpha_t \circ \phi^X_s(x)  \\
&-  \int_0^t    [(\frac{X\alpha}{\alpha}-1) (\frac{X\alpha}{\alpha}) - X (\frac{X\alpha}{\alpha})] 
\circ h^\alpha_\tau \circ \phi^X_s(x) d\tau\,.
 \end{aligned}
$$
\end{lemma}
\begin{proof} By Lemma \ref{lemma:tangent_flow} and by formula~(\ref{eq:velocityfunction})
 it follows that
$$
\frac{\partial v_t}{\partial s}(x,s)= \int_0^t   [v_\tau(x,s) U_\alpha + X] 
(\frac{X\alpha}{\alpha}) \circ h^\alpha_\tau \circ \phi^X_s(x) d\tau\,.
$$
By integration by parts we also have
$$
\begin{aligned}
&\int_0^t   v_\tau(x,s) U_\alpha(\frac{X\alpha}{\alpha}) \circ h^\alpha_\tau \circ \phi^X_s(x) d\tau 
= \int_0^t   v_\tau(x,s)  \frac{d}{d\tau} (\frac{X\alpha}{\alpha}) \circ h^\alpha_\tau \circ \phi^X_s(x) d\tau  \\
&= v_t(x,s) (\frac{X\alpha}{\alpha}) \circ h^\alpha_t \circ \phi^X_s(x)  -  
 \int_0^t    \frac{dv_\tau}{d\tau}(x,s) (\frac{X\alpha}{\alpha}) \circ h^\alpha_\tau \circ \phi^X_s(x) d\tau
 \\ &= v_t(x,s) (\frac{X\alpha}{\alpha}) \circ h^\alpha_t \circ \phi^X_s(x)  -  
 \int_0^t    [(\frac{X\alpha}{\alpha}-1) (\frac{X\alpha}{\alpha})] \circ h^\alpha_\tau \circ \phi^X_s(x) d\tau\,,
 \end{aligned}
$$
as claimed in the statement of the lemma.
\end{proof}
By the Sobolev embedding theorem, we have the following estimate:
\begin{lemma} 
\label{lemma:velocityderivativebound}
For any $r>7/2$, there exists a constant $C_r(\alpha)>0$  such that the following holds. Let 
$\alpha \in W^r(M)$. For all $(x,s)\in M\times [0,\sigma]$ and for all $t > 0$,
$$
\vert \frac{\partial v_t}{\partial s}(x,s) \vert  \leq  C_r (\alpha) \, t \,.
$$
\end{lemma}

For any $r>7/2$, let $C_r(\alpha)>0$ be the constant of 
Lemma~\ref{lemma:velocityderivativebound} and let
\begin{equation}
\sigma_r(\alpha) := \frac{1}{C_r(\alpha)}  \,>\,0\,.
\end{equation}
By a bootstrap argument we derive the following bound.

\begin{lemma}
\label{lemma:mixingbound}
For any $r>7/2$ and  for any $\sigma \in (0, \sigma_r(\alpha))$, there
exist a time $t_{r,\sigma}(\alpha)>0$ and a constant $C_{r,\sigma}(\alpha)>0$  such that, 
for any continuous function $f$ on $M$, for any $x\in M$ and for all $t> t_{r,\sigma}(\alpha)$,
$$
\sup_{ S\in [0, \sigma]} \vert \int_0^S  f\circ h^\alpha_t \circ \phi^X_s(x) ds \vert
\leq C_{r,\sigma}(\alpha)  \sup_{ S\in [0, \sigma]} \vert \frac{1}{t} \int_{\gamma^S_{x,t}} 
 f \hat U_\alpha \vert  \,.
$$
\end{lemma}
\begin{proof}
By Lemmas~\ref{lemma:mixingformula2} and~\ref{lemma:velocityderivativebound}, it follows
that 
$$
\begin{aligned}
&\sup_{ S\in [0, \sigma]} \vert \int_0^S  f\circ h^\alpha_t \circ \phi^X_s(x) ds \vert \leq 
\sup_{ S\in [0, \sigma]} \vert \frac{1}{t} \int_{\gamma^S_{x,t}}  f \hat U_\alpha \vert  \\
&+[ \max_{x\in M} \vert \frac{v_t(x,\sigma)}{t} +1\vert + C_r(\alpha) \sigma] 
\sup_{ S\in [0, \sigma]} \vert \int_0^S  f\circ h^\alpha_t \circ \phi^X_s(x) ds \vert \,.
\end{aligned}
$$
By definition $C_r(\alpha) \sigma<1$ for any $\sigma \in (0, \sigma_r(\alpha))$, hence
by Lemma~\ref{lemma:velocitybound} there exists  $t_{r,\sigma}(\alpha)>0$ such that,
for all $t> t_{r,\sigma}(\alpha)$, 
$$
\max_{x\in M} \vert \frac{v_t(x,\sigma)}{t} +1\vert  + C_r(\alpha) \sigma <1 \,.
$$
Thus we conclude that the statement holds if  we set
$$
 C_{r,\sigma}(\alpha) := [1-\vert \max_{x\in M}\frac{v_t(x,\sigma)}{t} +1\vert- C_r(\alpha) \sigma]^{-1}\,.
$$
\end{proof}

We recall that  the path $\gamma^\sigma_{x,t}$ is contained in a single leaf of the weak-stable
foliation of the geodesic flow. By Lemma~\ref{lemma:velocity},  its length is estimated below.
\begin{lemma} 
\label{lemma:length} 
For all $r>5/2$ and for all $\alpha \in W^r(M)$, there exists a  constant $C'_r(\alpha)>0$ such that, 
for all $\sigma >0$ and for all $(x,t)\in M\times \R^+$,  
$$
 \int_{\gamma^\sigma_{x,t}}  \vert \hat X \vert  \leq C'_r (\alpha)  \sigma \quad \text{ and } \quad
 \int_{\gamma^\sigma_{x,t}}  \vert \hat U\vert   \leq  C'_r (\alpha) \sigma t\,. 
$$
\end{lemma}
\begin{proof}  By Lemma~\ref{lemma:velocity} we have
$$
 \int_{\gamma^\sigma_{x,t}}  \vert \hat X \vert =  
\int_0^\sigma  ds    \quad \text{ and } \quad
 \int_{\gamma^\sigma_{x,t}}  \vert \hat U\vert  =  
 \int_0^\sigma   \frac{\vert v_t(x,s)\vert} {(\alpha \circ h^\alpha_t \circ \phi^X_s)(x)}  ds \,,
$$
hence the first estimate  in the statement is immediate, while the second estimate 
follows from the uniform linear bound on $\vert v_t(x,s)\vert$ established in  
Lemma~\ref{lemma:velocitybound}.
\end{proof}

From the results of \cite{BuFo} on the asymptotics of ergodic averages for the horocycle flow
(see in particular \cite{BuFo}, Theorem 1.3) we then derive the following bound:

\begin{lemma}
\label{lemma:BFbound}
Let $r > 11/2$. For any $\alpha \in W^r(M)$ and for any $\sigma>0$, there
exists a constant $C_{r,\sigma}(\alpha)>0$ such that the following holds. For any zero-average function 
$f\in W^r(M)\cap L^2_0(M, \text{\rm vol}_\alpha)$,  for all $x\in M$, for all $S\in (0,\sigma]$,  and for all 
$t>0$,
$$
\vert \int_{\gamma^S_{x,t}}  f \hat U_\alpha \vert \leq C_{r,\sigma} (\alpha) \Vert f\Vert_r \, 
 (S t)^{\frac{1+\nu_0}{2}} [1+ \log^+(St)]^{\epsilon_0}  \,.
$$
\end{lemma}
\begin{proof} Precise bounds on the integrals
$$
\int_{\gamma^S_{x,t}}  f \hat U_\alpha =  \int_{\gamma^S_{x,t}}  (\alpha f )\,\hat U  \,.
$$
can be derived from~\cite{BuFo}, Theorem 1.3, if the function $\alpha f \in W^r(M)$ is supported on irreducible unitary components of the the principal and complementary series. 

By Lemma~\ref{lemma:length}, it follows from \cite{BuFo}, Theorem 1.3,  that there exists a constant 
$C_{r,\sigma}(\alpha)>0$ such that,  for all $S\in (0,\sigma]$ and for all $t> 0$,
$$
\vert \int_{\gamma^S_{x,t}}  f \hat U_\alpha \vert  \leq C_{r,\sigma}(\alpha) \Vert f \Vert_r
 (S t)^{\frac{1+\nu_0}{2}} [1+\log ^+(S t)]^{\epsilon_0}\,.
$$
In case $\alpha f\in W^r(M)$ is supported on irreducible unitary components of the discrete series, 
since the path $\gamma^\sigma_{x,t}$ is contained in a leaf of the weak stable foliation of the geodesic flow (tangent to the integrable distribution $\{X,U\}$) it follows from the methods of \cite{BuFo} that 
$$
\vert \int_{\gamma^S_{x,t}}  f \hat U_\alpha \vert \leq C_{r,\sigma} ( \alpha) \Vert f\Vert_r 
(1+  \int_{\gamma^S_{x,t}}  \vert \hat X \vert )  \log ( 1 +  \int_{\gamma^S_{x,t}}  \vert \hat U\vert  )\,.
$$
By the above formula and by Lemma~\ref{lemma:length} the argument 
is completed.
\end{proof}

\begin{lemma} 
\label{lemma:mixingbound1}
Let $r > 11/2$. For any $\alpha \in W^r(M)$ and  $\sigma \in (0, \sigma_r(\alpha))$, there
exists a constant $C_{r,\sigma}(\alpha)>0$ such that the following holds. For any zero-average function 
$f\in W^r(M)\cap L^2_0(M, \text{\rm vol}_\alpha)$,  for all $x\in M$, for all $S \in (0,\sigma]$,  and for all $t>t_{r,\sigma}(\alpha)$,
$$
\vert \int_0^S (f\circ h^\alpha_t \circ \phi^X_s)(x) ds \vert \leq C_{r,\sigma} (\alpha) \Vert f\Vert_r \, 
 (S t)^{\frac{1+\nu_0}{2}} [1+ \log^+(St)]^{\epsilon_0} / t \,.
$$
\end{lemma}
\begin{proof} By the bounds proved in Lemma~\ref{lemma:velocitybound} and  Lemma~\ref{lemma:mixingbound}, the statement follows immediately from the above Lemma~\ref{lemma:BFbound} .
\end{proof}

By the invariance of the the standard volume under the geodesic flow and by integration by parts 
we derive the following formula.

\begin{lemma} 
\label{lemma:int_parts}
Let $\sigma >0$. For all $f \in L^2(M)$ and for all $g\in L^2(M)$ which belong to the 
maximal domain $\text{ \rm dom}(\mathcal L_X) \subset L^2(M)$ of the densely defined
Lie derivative operator $\mathcal L_X:L^2(M) \to L^2(M)$, for all $t\in \R$,
$$
\begin{aligned}
<f\circ h^\alpha_t, g> &= \frac{1}{\sigma}<  \int _0^\sigma f\circ h^\alpha_t \circ \phi^X_s ds, g \circ \phi^X_\sigma> \\
& - \frac{1}{\sigma}
\int _0^\sigma < \int _0^S f\circ h^\alpha_t \circ \phi^X_s ds, \mathcal L_X g \circ \phi^X_S>dS
\end{aligned}
$$
\end{lemma}
It is immediate to derive from Lemma~\ref{lemma:mixingbound1} and from Lemma
\ref{lemma:int_parts} estimates on the decay of correlations for sufficiently smooth functions in Theorem \ref{thm:mixingbound2}.

\begin{proofof}{Theorem}{thm:mixingbound2} 
By Lemma~\ref{lemma:mixingbound1} and Lemma~\ref{lemma:int_parts}, 
there exists a constant $C'_r(\alpha)>0$ such that
\begin{equation}
\label{eq:standardL2}
\vert <f\circ h^\alpha_t, g>  \vert \leq  C'_r(\alpha) \Vert f \Vert_r \Vert g\Vert_X\, 
t^{-\frac{1-\nu_0}{2}} (1+\log t)^{\epsilon_0} \,. 
\end{equation}
Finally, by taking into account that, for all $f$, $g \in W^r(M)$ and all $t \geq 0$, 
$$
<f\circ h^\alpha_t, g>_{L^2(M, \text{\rm vol}_\alpha)} = <f\circ h^\alpha_t,  \alpha g>\,,
$$
the theorem follows from the estimate in formula~(\ref{eq:standardL2}). In fact, by the Sobolev
embedding theorem, for all $g\in \text{\rm dom}(\mathcal L_X)$, 
$$
\Vert \alpha g \Vert_X   \leq  2 \Vert \alpha\Vert_r   \Vert g \Vert_X \,.
$$

\end{proofof}

\section{Twisted ergodic integrals: square-mean estimates}\label{sec:L2estimates}
We prove below $L^2$ bounds on twisted ergodic integrals of smooth 
functions, that is, on integrals of the form
$$
\int_0^{\mathcal T} w(t)  f\circ h^\alpha_t(x) dt  
$$
for any function $w\in L^\infty(\R, \C)$ and for any sufficiently smooth function
$f$ on $M$. In the next sections we will derive from these bounds 
estimates on spectral measures of smooth time-changes of the horocycle
flow.  The relevant twisted integrals are those with twist function equal to an exponential
function which are related to the Fourier transforms of spectral measures. 

We remark that the question on optimal {\it uniform} bounds for  twisted
ergodic integrals (with an exponential twist) is open even for the classical horocycle flow. 
Non-optimal, polynomial bounds can be derived from estimates on the rate of equidistribution 
and on the rate of mixing by an argument due to A.~Venkatesh \cite{Ve}. Such uniform bounds are
closely related to estimates on the rate of equidistribution of time-$\mathcal T$ maps of
horocycle flows. 

For our main results on spectral measures of time-changes, uniform bounds on twisted
ergodic integrals are not needed. In fact, the key step is to prove square-mean estimates 
of the type below.

\begin{lemma}
\label{lemma:L2twisted_est}
Let $r > 11/2$ and let $\alpha \in W^r(M)$. There exists  $\sigma_r(\alpha)>0$ such that
for all $\sigma \in (0, \sigma_r(\alpha))$ the following holds. There exists a constant $C_{r,\sigma}(\alpha)>0$ such that for any bounded weight function $w\in L^\infty(\R^+, \C)$,  for  any  continuos function $f\in \text{ \rm dom}(X)\subset L^2(M, \text{\rm vol}_\alpha)$ and for all $\mathcal T>0$,
$$
\begin{aligned}
\Vert \int_0^{\mathcal T} w(t) & f\circ h^\alpha_t dt \Vert_{L^2(M, \text{\rm vol}_\alpha)}     \leq  
C_{r,\sigma}(\alpha) \vert w \vert_\infty \Vert f \Vert_X^{1/2}
 \\ &\times \left[   \vert f\vert_\infty^{1/2} \mathcal T + \int_0^{\mathcal T}   \int_0^t  \Vert  
 \sup_{S\in [0, \sigma]}  \vert  \int_{\gamma^S_{x,\tau}} f \hat U_\alpha \vert \Vert_0 
  \frac{d\tau dt}{\tau}  \right]^{1/2}.
 \end{aligned}
$$
\end{lemma}
\begin{proof} By  the invariance of the volume form under the reparametrized horocycle flow and by change of variables we have that 
$$
\Vert \int_0^{\mathcal T}  w(t) f\circ h^\alpha_t  dt \Vert^2_{L^2(M, \text{\rm vol}_\alpha)}
= 2 \Re <\int_0^{\mathcal T}  \int_0^t w(t)   \overline{ w(t-\tau)} f \circ h^\alpha_\tau d\tau \, dt , \alpha f >.
$$
For every fixed $t\in \R$, let $w_t \in C^0(\R, \C)$ be the bounded weight function
defined as
$$
w_t (\tau) :=  \overline{ w(t-\tau)} \,, \quad \text{ \rm for all } \tau \in\R\,.
$$
By the formula of Lemma~\ref{lemma:int_parts} we have that
$$
\begin{aligned}
 &<\int_0^{\mathcal T}  \int_0^t   w(t) w_t(\tau) f \circ h^\alpha_\tau d\tau \, dt , \alpha f > \\ &
 = \frac{1}{\sigma}  \int_0^{\mathcal T}  \int_0^t w(t)    w_t(\tau) < \int _0^\sigma f\circ h^\alpha_\tau
 \circ \phi^X_s ds, (\alpha f) \circ \phi^X_\sigma> d\tau dt \\
& - \frac{1}{\sigma}
   \int_0^{\mathcal T}   \int_0^t   w(t)   w_t(\tau)  \int _0^\sigma <\int _0^S f\circ h^\alpha_\tau \circ \phi^X_s ds, \mathcal L_X (\alpha f) \circ \phi^X_S> dS  dt d\tau \,.
\end{aligned}
$$
The statement of the lemma then follows  from Lemma~\ref{lemma:mixingbound} and from the estimate
$$
\Vert  \int_0^{\mathcal T}   \int_0^{t_{r,\sigma}(\alpha)}  \int _0^S   w(t)   w_t(\tau) f\circ h^\alpha_\tau \circ \phi^X_s ds\, dt\,d\tau \Vert_0  \leq  \vert w \vert^2_\infty  \vert f\vert_\infty  t_{r,\sigma}(\alpha) S
\mathcal T \,.
$$
\end{proof}

\begin{thm}
\label{thm:twisted_int_L2bound}
Let $r > 11/2$. For any $\alpha \in W^r(M)$, there exists a constant $C_r(\alpha)>0$ such that the following holds. For any bounded weight function $w\in L^\infty(\R^+, \C)$, for  any zero-average function $f\in W^r(M)\cap L^2_0(M, \text{\rm vol}_\alpha)$ and for all $\mathcal T>0$,
$$
\Vert \int_0^{\mathcal T}  w(t) f\circ h^\alpha_t dt \Vert_{L^2(M, \text{\rm vol}_\alpha)} \leq  C_r(\alpha)\vert w \vert_\infty \Vert f \Vert_r  \mathcal T ^{\frac{3+\nu_0}{4}} [1+\log^+ \mathcal T]^{\frac{\epsilon_0}{2}}\,.
$$
\end{thm}
\begin{proof} By the equidistribution estimates proved in Lemma~\ref{lemma:BFbound},  for any $\sigma >0$ there exists a constant $C_{r,\sigma}(\alpha)>0$ such that for all $f\in W^r(M)\cap L^2_0(M, \text{\rm vol}_\alpha)$ and for all $x\in M$ and all $\tau>0$, 
$$
 \sup_{S\in [0,\sigma]}  \vert  \int_{\gamma^S_{x,\tau}}  
f \hat U_\alpha \vert 
\leq C_{r,\sigma}(\alpha)  \Vert f \Vert_r
\tau^{\frac{1+\nu_0}{2}} [1+\log^+ \tau]^{\epsilon_0}\,.
$$
The statement of the theorem then follows from Lemma~\ref{lemma:L2twisted_est} by integration.
\end{proof}

We remark that more refined estimates can be proved for functions supported on finite
codimensional subspaces orthogonal to irreducible components of the complementary
series and for coboundaries. From the estimates for coboundaries, which will be fully carried
out below, we will deduce that all smooth time-changes of the horocycle flow have absolutely
continuous spectrum.

\section {Spectral theory}\label{sec:spectrum}
In this final section we state and prove spectral results for smooth time changes of horocycle flows. 
In Section \ref{subsec:local} we first show  a local estimate on spectral measures of smooth functions (see Theorem \ref{thm:spectral_measures}). 
 Exploiting the $L_2$-bounds established in the previous Section~\ref{sec:L2estimates}, in Section \ref{subsec:ac} we prove the absolute continuity of the spectrum  (with countable multiplicity) for all smooth time-changes 
of the classical horocycle flow (Theorem \ref{thm:abs_cont_spectrum}). Finally in Section \ref{subsec:spectral_type} we show  that the maximal spectral type is always
equivalent to Lebesgue (Theorem \ref{thm:Max_spec_type}). 

\subsection{Local estimates} \label{subsec:local}

Let $\mu_f$ denote the spectral measures of a function $f\in L^2(M,\text{\rm vol}_\alpha)$. 
We recall that  $\mu_f$ is a complex measure on the real line. The main result derived in this section is the following.

\begin{thm}
\label{thm:spectral_measures}
Let $r>11/2$ and let $\alpha\in W^r(M)$.  There exists a constant $C_r(\alpha)>0$ such that,
for any function $u \in W^{r+1}(M)$ and for any $\xi \in \R\setminus\{0\}$,
$$
\vert \mu_u ( \xi -\delta, \xi +\delta)  \vert  \leq  C_r(\alpha) \Vert u \Vert_{r+1}   \frac{\delta \vert \log  \delta \vert}{\xi^2}\,, \quad \text{ \rm for all }\, \delta \in (0,\vert \xi\vert/2)\,,
$$
hence the measure $\mu_u$ has local dimension  $1$ at all points $\xi \in \R\setminus\{0\}$, that is, 
$$
\lim_{\delta\to 0^+}  \frac{ \log\mu_u ( \xi -\delta, \xi +\delta) }{\log \delta} =1\,.
$$
\end{thm}
The above result implies that the local dimension of spectral measures of smooth functions
is everywhere equal to $1$, but it is off by a logarithmic term from the sharpest possible bound, which would imply 
that spectral measures of sufficiently smooth functions are absolutely continuous with bounded
densities. In the following section (\S \ref{subsec:ac}) we nevertheless show how one can derive absolutely continuity from the mean-square bounds for ergodic integrals of coboundaries.

We first estimate twisted ergodic integrals of coboundaries.
Let us assume that $f \in L^2(M)$ is a smooth coboundary for the time-change  $\{h^\alpha_t\}$ on 
$M$,  that is,  there exists a function $u\in W^r(M)$ such that
$$
f= U_\alpha u \,.
$$

\begin{lemma} 
\label{lemma:coboundaries1}
There exists a constant $C>0$ such that for all $\sigma >0$, for all continuous coboundaries 
$f= U_\alpha u$ with transfer function $u\in L^\infty(M)$ such that $Xu \in L^\infty(M)$,  for all $x\in M$ and all $t>0$,
$$
 \sup_{S\in [0,\sigma]} \vert   \int_{\gamma^S_{x,t}}  f \hat U_\alpha  \vert  \leq  C\max\{1, \sigma\} 
\max\{ \Vert u \Vert_{L^\infty(M)},   \Vert X u \Vert_{L^\infty(M)} \} \,.
$$
\end{lemma}
\begin{proof}
By the definition of the path $\gamma^\sigma_{x,t}$ in formula~(\ref{eq:pushforwardpaths}) 
and by Lemma~\ref{lemma:velocity} we have
\begin{equation}
\label{eq:coboundary_formula}
\begin{aligned}
 \int_{\gamma^S_{x,t}}  f \hat U_\alpha &=  \int_{\gamma^S_{x,t}} du -
 \int_{\gamma^S_{x,t}} Xu \hat X \\ & = u\circ h^\alpha_t \circ \phi^X_S(x) - u\circ h^\alpha_t(x)- 
 \int_0^S Xu \circ h^\alpha_t \circ \phi^X_s(x) ds \,,
 \end{aligned}
\end{equation}
hence the statement of the lemma follows.
\end{proof}

\begin{cor}
\label{cor:coboundaries}
Let $r>11/2$ and let $\alpha\in W^r(M)$. There exists a constant $C_r(\alpha)>0$ such that 
 for any bounded weight function $w\in L^\infty(\R^+, \C)$, for all coboundaries 
 $f= U_\alpha u$ with a transfer function $u\in W^{r+1}(M)$ and for all $\mathcal T>0$,
$$
\Vert \int_0^{\mathcal T}  w(t) f\circ h^\alpha_t dt \Vert_{L^2(M, \text{\rm vol}_\alpha)} \leq  C_r(\alpha)\vert w \vert_\infty \Vert u \Vert_{r+1}  \mathcal T^{1/2} (1+ \log^+ \mathcal T)^{1/2}\,.
$$
\end{cor}
\begin{proof} By Sobolev embedding theorem if  $u\in W^{r+1}(M)$ then $u$, $Xu\in L^\infty(M)$ 
and the following estimate holds: there exists a constant $C_r>0$ such that
$$
\max\{ \Vert u \Vert_{L^\infty(M)},   \Vert X u \Vert_{L^\infty(M)} \} \leq C_r \Vert u\Vert_{r+1}\,.
$$
The statement then follows by integration from Lemma~\ref{lemma:L2twisted_est} and 
Lemma~\ref{lemma:coboundaries1}. 
\end{proof}

\begin{proofof}{Theorem}{thm:spectral_measures} 
By the spectral theorem, for any 
function $f\in L^2(M, \text{\rm vol}_\alpha)$ and  any $\xi \in \R$ ,
\begin{equation}
\label{eq:spectral_thm}
\begin{aligned}
 \Vert   \frac{e^{i (\xi + \eta)\mathcal T} -1}{i (\xi + \eta)}   \Vert^2_{L^2_\eta(\R, d\mu_f)} 
 &=  \Vert  \int_0^\mathcal T e^{i (\xi + \eta)t} dt   \Vert^2_{L^2_\eta(\R, d\mu_f)} \\
&= \Vert \int_0^{\mathcal T} e^{i \xi t} f\circ h^\alpha_t  dt  \Vert_0^2\,.
\end{aligned}
\end{equation}
By a simple computation, there exists a constant $C>0$ such that
\begin{equation}
\label{eq:meas_bound}
\begin{aligned}
 \mu_f ( \xi -\frac{1}{\mathcal T}, \xi + \frac{1}{\mathcal T}) &\leq 
C \int_{-\frac{1}{\mathcal T}}^{\frac{1}{\mathcal T}}  \vert 
\frac{e^{i (\xi + \eta)\mathcal T} -1}{i (\xi + \eta)\mathcal T} \vert^2  d\mu_f(\eta)  \\
 &\leq \frac{C}{\mathcal T^2} \Vert   \frac{e^{i (\xi + \eta)\mathcal T} -1}{i (\xi + \eta)}   
\Vert^2_{L^2_\eta(\R, d\mu_f)} \,.
\end{aligned}
\end{equation}
Let $u\in \text{\rm dom}(U_\alpha) \subset L^2(M, \text{\rm vol}_\alpha)$ and let $f:=U_\alpha u \in L^2(M, \text{\rm vol}_\alpha)$.  By the spectral theorem we have that
\begin{equation}
\label{eq:spectral_meas_der1}
d\mu_{f}(\xi) =  \xi^2  d\mu_ u(\xi) \,, \quad \text{ for all } \, \xi \in \R\,,
\end{equation}
hence there exists a constant $C'>0$ such that, for all $\xi \in \R\setminus \{0\}$,
\begin{equation}
\label{eq:spectral_meas_der2}
\mu_u ( \xi -\delta, \xi + \delta)  \leq   C'  \frac{\mu_{f} ( \xi -\delta, \xi + \delta)}{\xi^2}\,,
\quad \text{ \rm for all } \delta \in (0, \vert \xi \vert/2)\,.
\end{equation}
By formulas~(\ref{eq:spectral_thm}) and~(\ref{eq:meas_bound}) and by Corollary~\ref{cor:coboundaries}, it follows that there exists a constant $C_r(\alpha)>0$ such that, for all functions 
$u\in W^{r+1}(M)$ and for all $\mathcal T>0$, 
\begin{equation}
\label{eq:coboundary_meas_est}
 \mu_{f} ( \xi -\frac{1}{\mathcal T}, \xi + \frac{1}{\mathcal T}) \leq 
 C_r(\alpha) \Vert u\Vert_{r+1} \frac{ (1+ \log^+ \mathcal T)}{\mathcal T}\,.
\end{equation}
The statement of the theorem can readily be derived from  formulas~(\ref{eq:spectral_meas_der2}) 
and~(\ref{eq:coboundary_meas_est}).
\end{proofof}

\subsection{Absolute continuity}\label{subsec:ac}
We show in this section that 
that the  spectral measure $\mu_f$ of any function $f\in L^2(M,\text{\rm vol}_\alpha)$ is absolutely continuous with respect to the Lebesgue measure on the real line and hence that any smooth time-change of the horocycle flow has absolutely continuous spectrum. 

\begin{thm} 
\label{thm:abs_cont_spectrum}
Let $r>11/2$ and let $\alpha \in W^r(M)$. The time-change $\{h^\alpha_t\}$ of the (stable) 
horocycle flow $\{h^U_t\}$ with infinitesimal generator $U_\alpha:= U/\alpha$  has purely
absolutely continuous spectrum.
\end{thm}
The Theorem is derived below from the following estimate on decay of correlations of coboundaries. 

\begin{lemma} 
\label{lemma:coboundaries2}
Let $r>11/2$ and let $\alpha\in W^r(M)$. There exist constants $C_r(\alpha)>0$ and $t_r(\alpha)>0$ such that  the following holds. For any continuous coboundary $f= U_\alpha u$ with a transfer function 
$u\in L^\infty(M)$ such that $Xu\in L^\infty(M)$ and for any $g\in \text{\rm dom}(X)$,  for all 
$t>t_r(\alpha)$,
$$
\vert < f\circ h^\alpha_t, g> \vert  \leq  C_r(\alpha)
 \Vert u \Vert_{r+1}  \Vert g \Vert_X ( 1/t) \,.
$$
\end{lemma}
\begin{proof} The statement follows readily from Lemma~\ref{lemma:mixingbound}, Lemma~\ref{lemma:int_parts}  and Lemma~\ref{lemma:coboundaries1}.
\end{proof}

\begin{proofof}{Theorem}{thm:abs_cont_spectrum} Since coboundaries with smooth transfer functions are dense in 
$L^2(M,\text{\text vol}_\alpha)$, it is enough to prove that their spectral measures are
absolutely continuous. In fact, we will prove that for any coboundary $f=U_\alpha u$
with transfer function $u\in W^{r+1}(M)$ its spectral measure is absolutely continuous
with square-integrable density.  Let $\hat \mu_f $ be the Fourier transform of the spectral
measure $\mu_f$, which is the bounded function defined as follows:
$$
\hat \mu_f(t) :=  \int_\R   e^{i\xi t}   d\mu_f(\xi) \,, \quad \text{ for all } \, t\in \R\,.
$$
By definition of the spectral measures, for all $t\in \R$ we have the following identity
$$
\hat \mu_f(t)=  < f \circ h^\alpha_t, f >_{L^2(M, \text{\rm vol}_\alpha)} = < f \circ h^\alpha_t, \alpha f > \,.
$$
By Lemma~\ref{lemma:coboundaries2} we therefore have the following estimate: for any 
$t >  t_r(\alpha)$,
\begin{equation}
\label{eq:FT_bound}
\vert \hat \mu_f(t) \vert = \vert <  f \circ h^\alpha_t , \alpha f >\vert
\leq  C_r(\alpha) \Vert u \Vert_{r+1} \Vert \alpha f \Vert_X (1/t)\,.
\end{equation}
Since  $\hat \mu_f$ is a bounded function and it is symmetric, that is,  for all $t\in \R$, 
$$
\hat \mu_f(t) =  < f \circ h^\alpha_t, f >_{L^2(M, \text{\rm vol}_\alpha)} =  
< f , f \circ h^\alpha_{-t}>_{L^2(M, \text{\rm vol}_\alpha)} = \overline{\hat \mu_f(-t)}\,, 
$$
and since the function $1/t \in L^2\left((t_r(\alpha), +\infty), dt\right)$,
we have proved that the Fourier transform $\hat \mu_f\in L^2(\R, dt)$, which readily implies 
that the spectral measure $\mu_f$ is absolutely continuous with square-integrable Radon-Nikodym derivative.  In fact, there exists a constant  $C''_r(\alpha)>0$ such that the following estimate holds:
$$
\Vert \frac{D\mu_f}{D\xi} \Vert_{L^2(\R, d\xi)} =  \Vert \hat \mu_f \Vert_{L^2(\R, dt)}  \leq
C''_r(\alpha) \Vert u \Vert_{r+1} \Vert f \Vert_X \,.
$$
The proof of the theorem is complete. 
\end{proofof}

\subsection{Maximal spectral type} \label{subsec:spectral_type} Let us now  prove that the {\it maximal spectral type} of any 
smooth time-change of the classical horocycle flow is equivalent to Lebesgue.

\begin{lemma} 
\label{lemma:coboundaries3}
For any $\sigma \in (0, \sigma_r(\alpha))$ there exists a constant 
$C'_{r,\sigma}(\alpha)>0$ such that,  for all $x\in M$, for all $t>t_r(\alpha)$ and
for all functions $u\in C^1(M)$, we have 
$$
\vert \int_0^\sigma  U_\alpha u \circ h^\alpha_t \circ \phi^X_s(x) ds \vert  \leq C'_{r,\sigma}(\alpha) 
\max\{ \Vert u\Vert_{L^\infty(M)}, \Vert Xu\Vert_{L^\infty(M)} \} (1/t )  \,.
$$
\end{lemma}
\begin{proof} The statement follows from Lemma~\ref{lemma:mixingbound} and 
Lemma~\ref{lemma:coboundaries1}.  
\end{proof}

\begin{lemma}  
\label{lemma:orth_cond}
Let $r>11/2$ and let $\alpha \in W^r(M)$. Assume that the maximal spectral 
type of the time-change $\{h^\alpha_t\}$ is not Lebesgue. There exists a smooth non-zero function 
$w\in L^2(\R,dt)$ such that for all $x\in M$, for all $\sigma \in (0, \sigma_r(\alpha))$ and   
for all functions $u\in C^1(M)$ the following holds:
$$
\int_{\R} w(t) \int_0^\sigma  U_\alpha u \circ h^\alpha_t \circ \phi^X_s(x) ds \,dt =0
$$
\end{lemma}
\begin{proof} If the maximal spectral type is not Lebesgue, then the Lebesgue measure is not 
absolutely continuous with respect to the maximal spectral measure. Thus, there exists a compact 
set $A\subset \R$ such that $A$ has measure zero with respect to the maximal spectral measure 
of the flow $\{h^\alpha_t\}$, hence with respect to all its spectral measures, but $A$ has strictly positive Lebesgue measure.

Let $w\in L^2(\R)$ be the complex conjugate of the Fourier transform of the characteristic function 
$\chi_A$ of the set $A\subset \R$. For any pair of functions $f$, $g\in L^2(M)$ let $\mu_{f,g}$ denote 
the joint spectral measure.  By Theorem~\ref{thm:abs_cont_spectrum} the measure $\mu_{f,g}$ is
absolutely continuous with respect to Lebesgue. Whenever  $\mu_{f,g}$ has square-integrable
density we have 
\begin{equation}
\label{eq:orthogonality}
\int_{\R}  w(t) <f \circ h^\alpha_t, g >_{L^2(M, \text{\rm vol}_\alpha)}  dt = \int_{\R}  \chi_A(\xi) d \mu_{f,g} (\xi) =0\,.
\end{equation}
It follows from Lemma~\ref{lemma:coboundaries2} that, whenever $f=U_\alpha u$ is a coboundary 
with transfer function $u\in C^1(M)$, then  the Fourier transform of the spectral measure
 $\mu_{f,g}$, hence its density, is square-integrable, so that the identity in formula~(\ref{eq:orthogonality}) holds. 

From formula~(\ref{eq:orthogonality}) by translation under the geodesic flow and by 
integration we derive that, for any $\sigma >0$ and for any function $g\in W^r(M) \subset 
\text{\rm dom}(X)$, 
\begin{equation}
\label{eq:zeroint}
\int_0^\sigma \int_{\R}  w(t) <U_\alpha u \circ h^\alpha_t\circ \phi^X_s, g >_{L^2(M, \text{\rm vol}_\alpha)}  
dt ds= 0\,.
\end{equation}
By Lemma~\ref{lemma:coboundaries2} the double integral in formula~(\ref{eq:zeroint})
is absolutely convergent, hence 
\begin{equation}
\label{eq:zerointbis}
 < \int_{\R}  w(t)   \int_0^\sigma U_\alpha u \circ h^\alpha_t\circ \phi^X_s(\cdot)ds dt , g >_{L^2(M, \text{\rm vol}_\alpha)} = 0\,.
\end{equation}
It follows from Lemma~\ref{lemma:coboundaries3} that the function
 $$
 \int_{\R}  w(t)  \int_0^\sigma    U_\alpha u \circ h^\alpha_t\circ \phi^X_s(\cdot) ds dt
  $$
 is bounded on $M$, hence it vanishes identically by formula~(\ref{eq:zerointbis}) and 
 by density of the subspace $W^r(M) \subset L^2(M)$.

\end{proof}

\begin{lemma} 
\label{lemma:w_zero}
Let $w\in L^2(\R,dt)$  be a smooth function. Assume that for some $x\in M$, for some 
$\sigma\in (0, \sigma_r(\alpha))$ and  for all functions $u\in C^1(M)$, 
$$
\int_{\R} w(t) \int_0^\sigma  U_\alpha u \circ h^\alpha_t \circ \phi^X_s(x) ds \,dt =0\,.
$$
It follows that $w$ vanishes  identically. 
\end{lemma}
\begin{proof} Let us fix $x\in M$ and $\sigma>0$. For any given $T>0$ and $\rho>0$,
let  $E^T_{\rho,\sigma}$ be the flow-box for the time-change $\{h^\alpha_t\}$ defined as follows:
\begin{equation}
\label{eq:flow_box}
E^T_{\rho,\sigma} (r, s, t) =        (h^\alpha_t \circ \phi^X_s \circ h^V_r)(x) \,, \,\, \text{ \rm for all } 
(r,s,t) \in (-\rho, \rho) \times (-\sigma, \sigma)\times (-T,T)\,.
\end{equation}
Since the horocycle flow has no periodic orbits, it never returns to any given geodesic segment, 
hence for any $\sigma>0$ and any $T>0$ there exists $\rho>0$ such that $E^T_{\rho,\sigma}$ 
is an  embedding.  For any $\chi\in C^\infty(-1,1)$ and any $\psi \in C^\infty_0( -T, T)$, let
$$
\tilde u (r,s,t) :=  \chi(\frac{r}{\rho}) \chi(\frac{s}{\sigma}) \psi (t) \,,
 \quad \text{ \rm for all }  (r,s,t) \in (-\rho, \rho) \times (-\sigma, \sigma)\times (-T,T)\,.
$$
Let $u \in C^\infty(M)$ be the function defined as $u=0$ on $M\setminus
\text{\rm Im}(E^T_{\rho,\sigma})$ and as
\begin{equation}
\label{eq:u_def}
(u\circ   E^T_{\rho,\sigma}) (r, s, t) :=   \tilde u (r,s,t) \,,  \,\, \text{ \rm for all } 
 (r,s,t) \in (-\rho, \rho) \times (-\sigma, \sigma)\times (-T,T)\,,
\end{equation}
on $\text{\rm Im}(E^T_{\rho,\sigma})$. We claim that the  following formulas hold:
\begin{equation}
\label{eq:derivatives}
\begin{aligned}
(U_\alpha u) \circ E^T_{\rho,\sigma} (r, s, t) &:=  \chi(\frac{r}{\rho}) \chi(\frac{s}{\sigma}) 
\frac{d\psi}{dt} (t)\,, \\  
(X u)\circ E^T_{\rho,\sigma} (r, s, t) &:= \sigma \chi(\frac{r}{\rho})  \frac{d\chi}{ds} (\frac{s}{\sigma}) 
\psi (t)  \\ &-  v_t(h^V_r(x),s)    \chi(\frac{r}{\rho}) \chi(\frac{s}{\sigma}) 
\frac{d\psi}{dt} (t)\,.
\end{aligned}
\end{equation}
In fact, the above formulas are immediate consequence of the identities below.
Let $\mathcal T_{r,s} :\R^2 \to \R$ be the unique solution of the parametric Cauchy problem
 \begin{equation}
 \label{eq:Cauchy_problem}
\begin{cases}
\frac{\partial \mathcal T_{r,s}}{\partial S} (t,S) & = - v_{\mathcal T_{r,s}(t,S)}\left(  h^V_r(x), s+S\right)
 \,, \\
\mathcal T_{r,s}(t,0) &= t \,.
\end{cases}
\end{equation}
For any $(r, s, t)\in  (-\rho, \rho) \times (-\sigma, \sigma)\times (-T,T)$ there exists $(\tau_0,S_0)
\in \R^+$ such that, for all $(\tau,S) \in (-\tau_0, \tau_0)\times (-S_0, S_0) $, the following holds:
\begin{equation}
\label{eq:flow_box_id}
\begin{aligned}
h^\alpha_\tau \circ E^T_{\rho,\sigma} (r, s, t) &= E^T_{\rho,\sigma} (r, s, t+\tau) \,, \\
\phi^X_S \circ E^T_{\rho,\sigma} (r, s, t) &= E^T_{\rho,\sigma} (r, s+S, \mathcal T_{r,s}(t,S) )\,.
\end{aligned}
\end{equation}
The first of the above formulas is an immediate consequence of the definition~(\ref{eq:flow_box})
of the flow-box map. The second formula follows from the  commutation relation:
\begin{equation}
\label{eq:comm_rel}
\phi^X_S \circ h^\alpha_t \circ \phi^X_s \circ h^V_r(x) =   
h_{\mathcal T_{r,s}(t,S)} \circ \phi^X_{s+S}  \circ h^V_r(x) \,, \quad \text{ for all } (t,s,S) \in \R^3\,.
\end{equation}
Let us prove the above commutation identity. For $S=0$ it holds for all $(r,s,t)\in \R^3$.
For all $(r,s)\in \R^2$ let $x_{r,s}:=( \phi^X_s \circ h^V_r)(x)$. 
By Lemma~\ref{lemma:tangent_flow}, by differentiation of equation~(\ref{eq:comm_rel}) 
with respect to $S\in \R$, we find
$$
\begin{aligned}
(X \circ \phi^X_S \circ h^\alpha_t) (x_{r,s})&= 
 \frac{ \partial \mathcal T_{r,s}}{\partial S} (t,S) \,(U_\alpha 
\circ h^\alpha_{\mathcal T_{r,s}(t,S)} \circ \phi^X_{S}) (x_{r,s}) \\
&+   v_{\mathcal T_{r,s}(t,S)}(x_{r,s}, S)  \,(U_\alpha \circ  h^\alpha_{\mathcal T_{r,s}(t,S)} 
\circ \phi^X_{S})(x_{r,s})  \\
&+  (X  \circ  h^\alpha_{\mathcal T_{r,s}(t,S)} \circ \phi^X_{S}) (x_{r,s}) \,.
\end{aligned}
$$
Since the above equation holds by the definition of the function $\mathcal T_{r,s}$ in 
formula~(\ref{eq:Cauchy_problem}),  the commutation relation~(\ref{eq:comm_rel}) 
holds as well. We have thus proved the flow-box identities~(\ref{eq:flow_box_id}) from which the
differentiation formulas~(\ref{eq:derivatives})  follow immediately.

\smallskip
For any $\rho$, $\sigma>0$, let $T_{\rho,\sigma} >0$ be defined as follows:
\begin{equation}
T_{\rho,\sigma} := \min \{ \vert t \vert  > T  \,\vert\,  
 \cup_{s\in [-\sigma, \sigma]} (h^\alpha_t \circ \phi^X_s)(x) \cap
 \text{ \rm Im} (E^T_{\rho,\sigma}) \not= \emptyset \}\,.
\end{equation}
Since the horocycle flow  never returns to any given geodesic segment, for
every fixed $\sigma>0$, the following holds:
\begin{equation}
\label{eq:T_rho_div}
\lim_{\rho\to 0^+} T_{\rho,\sigma}  = + \infty\,.
\end{equation}
By assumption and by formula~(\ref{eq:derivatives}) we have 
\begin{equation}
\label{eq:orth_cond}
\begin{aligned}
 \chi(0) &\left(\int_0^\sigma \chi(\frac{s}{\sigma}) ds \right) \left( \int_{-T}^T  w(t) 
\frac{d\psi}{dt} (t) \,dt  \right) \\ &+ \int_{\R\setminus [-T_{\rho,\sigma}, T_{\rho,\sigma}]}  w(t)  
\int_0^\sigma U_\alpha u \circ h^\alpha_t\circ \phi^X_s(x) ds dt  \,=\,0\,.
\end{aligned}
\end{equation}
We claim that the following holds: for any fixed $\sigma\in (0, \sigma_r(\alpha))$ and $T>0$, 
\begin{equation}
\label{eq:lim_zero}
\lim_{\rho\to 0^+}  \int_{\R\setminus [-T_{\rho,\sigma}, T_{\rho,\sigma}]}  w(t)  
\int_0^\sigma U_\alpha u \circ h^\alpha_t\circ \phi^X_s(x) ds dt  \,=\,0\,.
\end{equation}
Since the function $u\in C^\infty(M)$, by Lemma~\ref{lemma:coboundaries3}, combined
with a trivial estimate for $0\leq t \leq t_r(\alpha)$,
there exists  a constant $C'_{r,\sigma}(\alpha)>0$ such that, for all $x\in M$ and for 
all $\sigma\in (0, \sigma_r(\alpha))$, we have
\begin{equation}
\label{eq:L2bound1}
\begin{aligned}
\Vert \int_0^\sigma U_\alpha u \circ h^\alpha_t\circ \phi^X_s(x) ds &\Vert_{ L^2(\R, dt)} \\
\leq  C'_{r,\sigma}(\alpha) &\max\{ \Vert u\Vert_{L^\infty(M)},  \Vert Xu\Vert_{L^\infty(M)}, 
\Vert U_\alpha u\Vert_{L^\infty(M)}\}\,.
\end{aligned}
\end{equation}
By the definition of the function $u\in C^\infty(M)$ (see formula~(\ref{eq:u_def})), by the
trivial estimate $\vert v_t\vert \leq C_\alpha t \leq C_\alpha T$ and by the
estimates in formula~(\ref{eq:derivatives}) we also have
\begin{equation}
\label{eq:L2bound2}
\begin{aligned}
 \max\{ \Vert u\Vert_{L^\infty(M)},  &\Vert Xu\Vert_{L^\infty(M)}, \Vert U_\alpha u\Vert_{L^\infty(M)}  \} \leq 
 C'_{r,\sigma}(\alpha)\max\{1, T\}  \\ &\times {\max}^2\{\Vert \chi \Vert_{L^\infty(\R)}, \Vert \chi' \Vert_{L^\infty(\R)} \} \max\{\Vert \psi \Vert_{L^\infty(\R)}, \Vert \psi' \Vert_{L^\infty(\R)}\}\,.
 \end{aligned}
\end{equation}
In particular the above bound is uniform with respect to $\rho>0$. Thus the limit
in formula~(\ref{eq:lim_zero}) follows from the uniform $L^2$ bound given by
 formulas~(\ref{eq:L2bound1}) and~(\ref{eq:L2bound2}). 

\smallskip
Since  formulas~(\ref{eq:orth_cond}) and~(\ref{eq:lim_zero}) hold for
 all functions $\chi\in C^\infty_0(-1,1)$, for all $T>0$ and for all
 functions  $\psi \in C^\infty_0(-T,T)$, it follows that 
$$
 \int_\R  w(t)  \frac{d\psi}{dt} (t) \,dt  =0\,,  \quad \text{ \rm for all } \psi \in C^\infty_0(\R)\,,
$$
hence the function $w\in L^2(\R,dt)$ is a constant (necessarily equal to zero).
\end{proof}

By Lemma~\ref{lemma:orth_cond}  and Lemma~\ref{lemma:w_zero} we derive our
conclusive spectral result.  

\begin{thm} 
\label{thm:Max_spec_type}
Let $r>11/2$ and let $\alpha \in W^r(M)$.  The maximal spectral type of the
time-change $\{h^\alpha_t\}$ of the (stable) horocycle flow $\{h^U_t\}$ with infinitesimal generator
$U_\alpha:= U/\alpha$ is equivalent to Lebesgue.
\end{thm}

\section*{Acknowledgements}
We are extremely grateful to A.~Katok who in several occasions raised  the question 
of whether smooth time-changes of horocycle flows have absolutely continuous spectrum,
thereby providing the main motivation for this work. After we answered this question in
the affirmative, he asked whether we could prove that the maximal spectral type is Lebesgue,
motivating a significant improvement of the paper. We are also very grateful to L.~Flaminio
who informed us of the crucial work of B.~Marcus \cite{Ma} and thus put us on the right track.

\bibliographystyle{plain}

\end{document}